\documentclass[12pt]{amsart}

\usepackage[colorinlistoftodos,bordercolor=orange,backgroundcolor=orange!20,linecolor=orange,textsize=scriptsize,disable]{todonotes}

\usepackage{graphicx}      %
\usepackage[margin=1in]{geometry} 
\usepackage{amsmath}
\usepackage{amssymb}
\usepackage{algorithm}  
\usepackage{algorithmicx}  
\usepackage{algpseudocode}  
\usepackage{graphics} 
\usepackage{float}
\usepackage{subfigure}
\usepackage{frenchineq}
\newcommand{\veryshortrightarrow}[1][3pt]{\mathrel{%
   \hbox{\rule[\dimexpr\fontdimen22\textfont2-.2pt\relax]{#1}{.4pt}}%
   \mkern-4mu\hbox{\usefont{U}{lasy}{m}{n}\symbol{41}}}}
\newcommand{\veryshortleftarrow}[1][3pt]{\mathrel{%
   \hbox{\usefont{U}{lasy}{m}{n}\symbol{40}}%
   \mkern-4mu\hbox{\rule[\dimexpr\fontdimen22\textfont2-.2pt\relax]{#1}{.4pt}}%
}}
\makeatletter

\setbox0\hbox{$\xdef\scriptratio{\strip@pt\dimexpr
		\numexpr(\sf@size*65536)/\f@size sp}$}

\newcommand{\scriptveryshortarrow}[1][3pt]{{%

		\hbox{\rule[\scriptratio\dimexpr\fontdimen22\textfont2-.2pt\relax]
			{\scriptratio\dimexpr#1\relax}{\scriptratio\dimexpr.4pt\relax}}%
		\mkern-4mu\hbox{\let\f@size\sf@size\usefont{U}{lasy}{m}{n}\symbol{41}}}}

\makeatother

\newcommand{\fromsource}{_{\veryshortrightarrow}}
\newcommand{\todest}{_{\veryshortleftarrow}}

\newcommand{\bO}{\mathcal{O}}
\newcommand{\Od}{\mathcal{O}}
\newcommand{\Id}{\mathcal{I}}
\newcommand{\Ad}{{\mathcal A}}

\newcommand{\R}{\mathbb{R}}

\newcommand{\F}{{\mathcal{F}}}

\newcommand{\Rm}{\R_{\max}}
\newcommand{\Rmb}{\overline{\R}_{\max}}
\newcommand{\W}{\mathcal{W}}
\newcommand{\Z}{\mathcal{Z}}
\newcommand{\X}{\mathcal{X}}
\newcommand{\Y}{\mathcal{Y}}
\newcommand{\U}{\mathcal{U}}
\newcommand{\Ba}{\textit{Base}}
\newcommand{\Te}{\textit{Test}}
\newcommand{{\Mw}}{\mathcal{M}}
\newcommand{\Nw}{\mathcal{N}}
\def\<#1,#2>{\langle #1, #2\rangle}
\newcommand{\Ss}{\mathcal{S}}
\newcommand{\T}{\mathcal{T}}

\newcommand{\Space}{\mathcal{C}_{spa}}

\newenvironment{skproof}{{\noindent\it Sketch of Proof.}}{\hfill $\square$\par}

\newtheorem{theorem}{Theorem}[section]
\newtheorem{proposition}{Proposition}[section]

\theoremstyle{definition}
\newtheorem{definition}{Definition}[section]
\newtheorem{assumption}{Assumption}[section]
\theoremstyle{remark}



\begin{document}
 \title[Adaptive Multi-Level Max-Plus Method]{An Adaptive Multi-Level Max-Plus Method for Deterministic Optimal Control Problems}
 
 \author{Marianne Akian}
 \author{St\'ephane Gaubert}
 \address[Marianne Akian and St\'ephane Gaubert]{Inria and CMAP, \'Ecole polytechnique, IP Paris, CNRS}
 \email{Marianne.Akian@inria.fr}
 \email{Stephane.Gaubert@inria.fr}
 \author{Shanqing Liu}
 \address[Shanqing Liu]{CMAP, \'Ecole polytechnique, IP Paris, CNRS, and Inria}
 \email{Shanqing.Liu@polytechnique.edu}

\begin{abstract}                %
  We introduce a new numerical method to approximate the solution of a finite horizon deterministic optimal control problem. We exploit two Hamilton-Jacobi-Bellman PDE, arising by considering the dynamics in forward and backward time. This allows us to compute a neighborhood of the set of optimal trajectories, in order to reduce the search space. The solutions of both PDE are successively approximated
  by max-plus linear combinations of appropriate basis functions, using a
  hierarchy of finer and finer grids. 
  We show that the sequence of approximate value functions obtained in
  this way does converge to the viscosity solution of the HJB equation in a
  neighborhood of optimal trajectories.
  Then, under certain regularity assumptions, we show that the number of arithmetic operations needed to compute an approximate optimal solution of a $d$-dimensional problem, up to a precision $\varepsilon$, is bounded by $O(C^d (1/\varepsilon) )$, for some constant $C>1$,
  whereas ordinary grid-based methods have a complexity in
  $O(1/\varepsilon^{ad}$) for some constant $a>0$. 
\end{abstract}

\maketitle

%
\todo[inline]{SG: I edited the end of the abstract to emphasize the benefit of plog growth  in the title, to emphasize the multi-level nature. Implemented.}
\todo[inline]{SG: what about saying ``An adaptative multi-level Max-Plus..''
  in the title, to emphasize the multi-level nature. Implemented}
\section{Introduction}
We are interested in numerically solving a finite horizon deterministic optimal control problem. Such a problem is associated to a Hamilton-Jacobi-Bellman (HJB) equation via the Bellman dynamic programming principle  (see for instance \cite{flemingsoner}).
The value function of this class of optimal control problems has been characterized as the solution of the associated HJB PDE in the viscosity sense (\cite{crandall1983viscosity}). Several classes of numerical methods have been proposed to solve such PDE. Among them, we mention the finite difference schemes introduced in \cite{crandall1984two}, and the semi-lagrangian schemes, studied in particular in (\cite{falcone1987numerical}, \cite{falcone2013semi}). %

More recently, max-plus based discretization schemes have been developed
by \cite{fleming2000max}, \cite{akian2008max}, \cite{Mc2007}, \cite{Mc2006}, \cite{qu2014contraction}, \cite{dower}. These methods take advantage of the max-plus linearity of the evolution semigroup of the HJB PDE, the so called {\em Lax-Oleinik semigroup}.
After a time discretization, this allows one to approximate the value function for a given horizon, by a supremum of appropriate basis functions, for instance quadratic forms. Such suprema are propagated by the action of the Lax-Oleinik semigroup, between two successive time steps.
\cite{Mc2007} showed that the max-plus based methods have the advantages to attenuate the \textit{curse-of-dimensionality} in some structured cases, including switched control problems, see also~\cite{eneaneyphys}, and~\cite{qu2014contraction} for further complexity results. Other attempts to reduce \textit{curse-of-di\-men\-sio\-na\-li\-ty} include the computation of the value function at one given point by constructing the grid from the possible trajectories and reducing the set of trajectories using Lipschitz continuity properties, together with the low dimensionality of the control set, like in \cite{alla2019efficient}, \cite{alla2020tree}, 
 and \cite{bokanowski2022optimistic}.

  In this paper, we address the curse-of-dimensionality issues with another approach. The main idea is to consider a hierarchy of finer and finer irregular grids, concentrated around optimal trajectories, thus allowing us to dynamically reduce the search space, while increasing the precision. This is achieved by considering a pair of HJB PDE, associated to two optimal control
  problems: one
  with a forward dynamics, fixed initial state and free final state,
  and a dual one, with a backward dynamics, fixed final state and free
  initial state. The value functions of these two PDE allow us
  to compute a family of nested neighborhoods of optimal trajectories.
  Then, we adaptively add new basis functions, from one grid level to the next one, to refine the approximation. These new basis functions are chosen to be concentrated near the optimal trajectories of the control problem, and the refined neighborhood
  of optimal trajectories is computed from the solutions of the two HJB PDE
  in the coarser grid.

  We show that using our algorithm, the number of basis functions needed to get a certain error $\varepsilon$ is considerably reduced. Indeed, for a $d$-dimensional problem, under certain regularity assumptions, we get a complexity bound
  of $C^d(1/ \varepsilon) $ arithmetic operations, for some constant $C>1$.
  This should be compared with methods based
  on regular grids, which yield complexity bounds
  of order $\bO(1/\varepsilon^{ad})$ in which $a>0$ depends
  on regularity assumptions and on the order of the scheme (see for instance~\cite{bardi2008optimal}). %
  With our adaptative method, the curse
  of dimensionality remains only present in the term $C^d$, in particular
  the complexity becomes linear in the bit-size of the numerical precision
  $\varepsilon$.
  
  The present work extends the idea of dynamic grid refinement, originally
  presented in\-~\cite{akian2023multi} to solve semi-Lagrangian discretizations
  of special, minimal time, problems. Here, we exploit max-plus approximations
  combined with direct methods, allowing a higher degree of accuracy,
  and we adress finite horizon problems with more general cost and dynamics structure.

\section{Optimal control problem, HJB equation, characterization of optimal trajectories}\label{sec_subdomain}
We intend to solve the following finite horizon deterministic optimal control problem:
\begin{equation}\label{problem}
	\max \left\{ \int_0^T \ell(x(s),u(s))ds + \phi_0(x(0))+\phi_T(x(T)) \right\}
\end{equation}
over the set of trajectories $(x(s),u(s))$ satisfying:
\begin{equation}\label{dynamic&constraint}
\left\{
\begin{aligned}
& \dot{x}(s) = f(x(s),u(s)) \ , \\
& x(s) \in X, \  u(s) \in U   \ ,
\end{aligned}	
	\right.
\end{equation}
for all $ s \in[0,T]$. Let us denote $v^*$ the maximum in \eqref{problem}. Here, $X \subset\R^d$, assumed to be bounded, is the state space and $U\subset \R^m$ is the control space. We further assume that the running cost $\ell: X \times U \mapsto \R$, the dynamics $f:X \times U \mapsto \R$, the initial and final cost $\phi_0,\phi_T:X\mapsto \R$ are sufficiently regular: bounded, continuous and Lipschitz w.r.t.\ all variables. 

A well known sufficient and necessary optimality condition for the above problem is given by the Hamilton-Jacobi-Bellman equation, which is deduced from the dynamic programming principle. Indeed, we consider the value function $v\todest$,
defined as follows,
for any $(x,t) \in X \times [0,T]$:
\begin{equation}\label{value_todest}
v\todest(x,t) = \sup \left\{ \int_t^T \ell(x(s),u(s))ds + \phi_T(x(T)) \right\}	 \ ,
\end{equation}
under the constraint \eqref{dynamic&constraint} with the initial state $x(t)=x$.
Here, the symbol $"\todest"$ indicates that $(x,t)$ is the {\em source},
so that the corresponding HJB PDE is of a backward nature.
Indeed, $v\todest$ is known to be the viscosity solution of the following HJB equation (see for instance \cite{flemingsoner}):
\begin{equation}\label{HJB_todest}
\left\{
\begin{aligned}
&-\frac{\partial v\todest}{\partial t} - H(x,\nabla v\todest) = 0, \ &(x,t)\in X \times[0,T] \ , \\
&v\todest(x,T) = \phi_T(x), \ &x \in X \ ,
\end{aligned}	
\right.	
\end{equation}
where $H(x,p) = \sup_{u \in U} \{ p \cdot f(x,u) + \ell(x,u)\}$ is the Hamiltonian of the problem. Once \eqref{HJB_todest} is solved, one can easily obtain the value of the original problem \eqref{problem} by further taking the maximum over $X$, i.e., 
\begin{equation}\label{opt_todest}
	v^* = \max_{x \in X} \{ \phi_0(x) + v\todest(x,0)  \} \ .
\end{equation}
We shall also use another, equivalent, optimality condition for problem \eqref{problem}, obtained by applying the dynamic programming principle in
the reverse direction. 
This leads us to consider the value function $v\fromsource$, such that
\begin{equation}\label{value_fromsource}
	v\fromsource(x,t) = \sup \left\{ \int_0^t \ell(x(s),u(s))ds + \phi_0(x(0)) \right\} \ ,
\end{equation}
under the same constraint \eqref{dynamic&constraint}, but with the final state $x(t)=x$.
The notation $"\fromsource"$ indicates that $(x,t)$ is now the destination.
Then, $v\fromsource$ is known to be the viscosity solution of the following HJB equation,
in forward time:
\begin{equation}\label{HJB_fromsource}
\left\{
\begin{aligned}
	&  \frac{\partial v\fromsource}{\partial t} - H(x, - \nabla v\fromsource) = 0, \ & (x,t) \in X \times [0,T] \ , \\
	& v\fromsource(x,0) = \phi_0(x) , \   & x \in X \ .
\end{aligned}
\right.	
\end{equation}
Once \eqref{HJB_fromsource} is solved, we can then get the maximum in \eqref{problem} by
\begin{equation}\label{opt_fromsource}
    v^* = \max_{x \in X} \{ \phi_T(x) + v\fromsource(x,T) \} \ .
\end{equation}
The two value functions $v_{\fromsource}$ and $v_{\todest}$ allow us to determine the points belonging to optimal trajectories:
\begin{definition}\label{def_opttra}
  We say that $x^*(\cdot)$ is an optimal trajectory of the optimal control problem~\eqref{problem} if there exists a control $u^*(\cdot)$ such that $(x^*(\cdot),u^*(\cdot))$
  achieves the maximum in~\eqref{problem},
  under the constraint \eqref{dynamic&constraint}.
  We assume that the set of optimal trajectories is non-empty,
and denote, for all $t\in [0,T]$:
	\begin{equation}\label{opt_set}
		\Gamma^*_t = \{x^*(t) \mid \text{ $x^*(\cdot)$ is an optimal trajectory }\} \ ,
	\end{equation}
and $\Gamma^*=\cup_{t\in [0,T]} \Gamma^*_t$.
\end{definition}
Then, we have the following result:
\begin{proposition}\label{pro_opt}
	\begin{equation}\label{char_opt_value}
		v^* = \sup_{x\in X} \{v\fromsource(x,t) + v\todest(x, t) \}, \ \forall t \in [0,T] \ .
	\end{equation}
	Moreover, for all  $t\in [0,T]$, the above supremum is achieved for some $x \in \Gamma^*_t$. Conversely, for all $x \in \Gamma^*_t$, %
the above supremum is achieved at point $x$.
\end{proposition}
	\begin{proof}
          The equality \eqref{char_opt_value} follows in a straightforward way from the definition of the value functions $v\todest, v\fromsource$ in \eqref{value_todest} and \eqref{value_fromsource}.
          Moreover, since there exists an optimal trajectory $x^*(\cdot)$, then the supremum in~\eqref{char_opt_value} is achieved at $x^*(t)\in \Gamma^*_t$, for all $t\in [0,T]$. Conversely, for all $x\in \Gamma^*_t$, there exists an optimal trajectory $x^*$ such that $x^*(t) = x$, and the supremum in~\eqref{char_opt_value} is achieved at $x=x^*(t)$.
		\end{proof}
For all $t\in [0,T]$, let us define the map $\F_v^t: X \mapsto \R$ by 
\begin{equation}\label{map_value}
	\F_v^t(x)=\F_v(x,t) = v\fromsource(x,t) + v\todest(x,t) \ .
\end{equation}
Consider for every $t\in[0,T]$ the subdomain $\Od_\eta^t \subset X$,
depending on a parameter $\eta>0$, and defined as follows:
\begin{equation}\label{subdomain}
\Od^t_\eta = \{ x\in X \mid \F_v^t(x) > \sup_{x\in X} \F_v^t(x) - \eta  \}	 \ .
\end{equation}
In fact, $\Od^t_\eta$ can be thought of as a $\eta-$neighborhood around the geodesic points at time $t$, $\Gamma^*_t$. We set $\Od_\eta := \{ (x,t) \mid x \in \Od^t_\eta \}$.
We intend to reduce the (state,time)-space $X\times [0,T]$ of our optimal control problem
to such an $\eta-$neighborhood. I.e., we replace the constraint~\eqref{dynamic&constraint}
by
\begin{equation}\label{constraint_reduced}
\left\{
\begin{aligned}
	& \dot{x}(s) = f(x(s),u(s)) \ , \\
	& x(s) \in \Od^s_\eta, \ u(s) \in U \ , 
\end{aligned}	
\right.
\end{equation}
for all $s\in[0,T]$. Let us denote $v^*_\eta$ the maximum of \eqref{problem} under the new constraint \eqref{constraint_reduced}. Then we have
\begin{proposition}\label{pro_value}
    $v^*_\eta = v^* \ .$	
\end{proposition}
\begin{proof}
	The inequality $v^* \geq v^*_\eta $ is  straightforward since $\Od_\eta^s \subset X$ for all $s\in[0,T]$. To show the reverse inequality, let us take an optimal trajectory $x^*(\cdot)$ for the original problem. Then, by the result of Proposition \ref{pro_opt}, we have $x^*(s)\in \Od^s_\eta, \forall s\in[0,T]$. Thus $v^*_\eta \geq v^*$ since $v^*$ is exactly the value of the integral in \eqref{problem} following the optimal trajectory $x^*(\cdot)$ .
\end{proof}
Proposition \ref{pro_value} indeed tells us that, to solve the problem \eqref{problem}, only %
   the $\eta-$neighborhood, $\Od_\eta$, around the optimal trajectory is relevant. In the following, we will focus on solving the problem $\eqref{problem}$ using an approximation of such a neighborhood.

\section{Propagation by Lax-Oleinik Semi-Groups and Max-Plus Approximation}
We denote by $S^\tau\todest$ the {\em Lax Oleinik semigroup} of \eqref{HJB_todest} with $\tau = T-t$,
i.e., the evolution semigroup of this PDE,
meaning that, for all $0\leq t\leq T$, $S\todest^\tau$ is the map sending the final cost function $\phi_T(\cdot)$ to the value function $v\todest(\cdot,\tau)$, so that the semi-group property $S^{\tau_1+\tau_2}=S^{\tau_1}\circ S^{
	\tau_2}$ is satisfied. In addition, the map $S^\tau\todest$ is
{\em max-plus linear}, meaning that for all $\lambda \in \R$ and for all functions
$\phi_T^1$ and $\phi_T^2:X\to\R$, we have:
\begin{equation}\label{mplinear}
\begin{aligned}
	&S^\tau\todest[\sup(\phi_T^1 , \phi_T^2)] = \sup(S^\tau\todest[\phi_T^1] , S^\tau\todest[\phi_T^2]) \ ,\\
	&S^\tau\todest[\lambda + \phi^1_T] = \lambda + S^\tau\todest[\phi_T^1] \ ,
\end{aligned}	
\end{equation}
where for any function $\phi$ on $X$, $\lambda+\phi$ is the function
$x\in X\mapsto \lambda+\phi(x)$
(see for instance \cite{fleming2000max}, \cite{akian2008max}, \cite{dower}).
Indeed, the property~\eqref{mplinear}
can be interpreted as the linearity in the sense of the max-plus semifield, which is the set
$\Rm: = \R\cup\{-\infty\}$  equipped with the addition $a \oplus b: = \max(a,b)$ and the multiplication $a \odot b:= a+b$, with $-\infty$ as the zero and $0$ as the unit.
Notice that the above properties hold, mutatis mutandis, for the evolution
operator of the dual equation~\eqref{HJB_fromsource}.
We will then briefly describe the approximation method based on the max-plus linearity  introduced in \cite{akian2008max}, which may be thought of as a max-plus analogue of the finite element methods.

Let us discretize the time horizon by $N = \frac{T}{\delta}$ steps. Denote $v^t\todest = v\todest(\cdot,t)$. By the semigroup property we have:
\begin{equation}\label{propagation}
v^{t-\delta}\todest = S^\delta \todest[v^t\todest], \ \forall \ t =  \delta, 2\delta, \dots , T \ , \quad v\todest^T=\phi_T\enspace .	
\end{equation}
Denote $\Rmb:=\Rm\cup\{+\infty\}$ the complete semiring extending $\Rm$, and let $\W$ be a complete $\Rm$-semimodule of functions $w: X \to \Rmb$, meaning that $\W$ is stable under taking the supremum of an arbitrary family of functions,
and by the addition of a constant, see~\cite{Mc2006,cgq02} for background.
We choose this semimodule $\W$ in such a way that $v^t\todest \in \W$ for all $t \geq 0$. In many applications, the value function $v^t$ is known to be
$c$-semiconcave for all $t\in [0,T]$,
and then $\W$ can be taken to  be the set of $c$-semiconcave functions,
which is a complete module, see~\cite{Mc2006,akian2008max}.
We also choose $\Z$, a complete $\Rm$-semimodule of test functions $z: X \mapsto \Rmb$. %
If the space of test functions $\Z$ is large enough,  \eqref{propagation} is equivalent to:
\begin{equation}\label{var_form}
	\< z , v^{t-\delta}\todest > = \< z , S^\delta\todest [v^t\todest] >  \;\forall t,\; \< z , v^T >= 	\< z , \phi_T > \; \forall z \in \Z ,
\end{equation}
where the max-plus scalar product of $u\in \W$ and $v\in \Z$
is defined by %
$\< u , v > = \sup_{x \in X} (u(x)+v(x)) \in \Rmb$.

Note that in the system~\eqref{var_form},  the unknown value
functions are elements of $\W$, therefore having an infinite number
of degrees of freedom, and that there are infinitely many equations
(one for each element $z\in\Z$).
Hence, we need to discretize this system.
To do so, we consider $\W^h \subset \W$, a semimodule generated by a finite family of basis functions $\{w_i\}_{1\leq i \leq p}$. %
The value function $v^t\todest$ at time $t$ is approximated by $v^{t,h} \todest\in \W^h$, that is:
\begin{equation}\label{mpapproxi}
v^{t,h}\todest %
:= \sup_{1\leq i \leq p} \{\lambda^t_i + w_i \} \ :\ 
x\mapsto \max_{1\leq i \leq p} \{\lambda^t_i + w_i(x) \} \ ,
\end{equation}
where $\{\lambda^t_i\}_{1 \leq i \leq p}$ is a family of scalars.
We then consider $\Z^h \subset \Z$, a semimodule generated by a finite family of test functions $\{ z_j \}_{1 \leq j \leq q}$, and, instead of requiring~\eqref{var_form} to hold for all $z\in \Z$, we only require that it holds for
generators, leading to a finite system of equations.
Therefore, the approximation $v^{t-\delta,h}\todest$ and $v^T$ should satisfy: %
 \begin{equation}\label{approt+delta}
     \< z_j , v^{t-\delta,h}\todest > = \< z_j , S^\delta\todest [v^{t,h}\todest] > , \ \< z_j , v^T >= \< z_j , \phi_T >  \ \forall j\ .
 \end{equation}
 It is a key property of max-plus algebra that a system of linear equations, even when the number of equations coincides with the number of degrees of freedom, and when the system is ``nonsingular'', may have no solution, so that the
 notion of solution must be replaced by a notion of maximal subsolution,
 which is always well posed.
 In particular, \eqref{approt+delta} may not have a solution. Hence, 
 we define $v^{t-\delta,h}\todest$ to be the maximal
 solution of the following system of inequalities:
 \begin{equation}\label{maxsolution}
	\< z_j , v^{t-\delta,h}\todest > \leq \< z_j , S^\delta\todest [v^{t,h}\todest] > , \ \< z_j , v^T >\leq  \< z_j , \phi_T >, \ \forall j \ .
\end{equation}
Let us denote $W_h: \Rm^p \mapsto \W$ the max-plus linear operator such that 
$W_h (\lambda) = \oplus_{1 \leq i\leq p} \{\lambda_i \odot w_i \}$,
and $Z^*_h :\W  \mapsto\Rm^q $ with $(Z^*_h(w))_j = \< z_j , w >, \forall 1 \leq j \leq q$.  %
Recall that, 
for every ordered sets $\Ss,\T$ and order preserving map $g: \Ss \mapsto \T$, the residuated map $g^{\#}$ is defined as $g^{\#}(t) = \max \{ s\in \Ss \mid g(s)\leq t \}$, when it exists. Max-plus linear operators have a residuated map. Moreover,
by \cite[Th.~1]{cohen1996kernels}, 
for all max-plus linear operators 
$B: \U \mapsto \X$, $C : \X \mapsto \Y$ over complete semimodules $\X,\Y, \U$, the operator $\Pi_B^C := B \circ (C \circ B)^{\#} \circ C  $ is a projector, and we have, %
for all $x \in \X$:
	\begin{equation}
	\Pi_B^C (x)= \max \{ y \in \mathrm{im} B \mid Cy \leq Cx \} \ .
	\end{equation}
Then, the approximations $v^{t,h}\todest$ can be expressed as follows.
\begin{proposition}[\cite{akian2008max}]\label{recursion_lambda}
Consider the maximal  $\lambda^t \in \Rm^p$ and $v^{t,h}\todest\in \W_h$,
 $t = 0, \delta$, $\dots , T$, 
such that 
$v^{t,h}\todest = W_h \lambda ^t$,
with $v^{t-\delta, h}\todest$, $t\geq \delta$, 
and $v^T$ solutions of \eqref{maxsolution}.
We have, 
\[ v^{t-\delta,h}\todest = S^{\delta,h}\todest[v^{t,h}\todest],
\; \text{where}\;     S^{\delta,h}\todest = \Pi^{Z^*_h}_{W_h} \circ S^\delta\todest \ ,\]
and %
\[		\left\{
			\begin{aligned}
				&\lambda^{t-\delta} = (Z^*_h W_h)^{\#} (Z^*_h S^\delta\todest W_h \lambda^t), \ \forall t = \delta, 2 \delta ,\dots, T \ , \\
				&\lambda^T = W_h^\#  \phi_T \ . 
			\end{aligned}
			\right.\]
\end{proposition}
The above formula can expressed using the linear operators $M_h:=Z^*_h W_h$ and $K_h:=Z^*_h S^\delta\todest W_h $, with entries:
\begin{equation}\label{mass_stiffness}
	\begin{aligned}
	&(M_h)_{j,i} = \< z_j , w_i > \ , \quad %
(K_h)_{j,i} = \< z_j , S^\delta\todest w_i > \ .	
	\end{aligned}
\end{equation}
The matrices $M_h$ and $K_h$ may be thought of as max-plus analogues
of the {\em mass} and {\em stiffness} matrices arising in the finite
element method, see~\cite{akian2008max}.
Computing $(M_h)_{j,i}$ is a convex programming problem, which can be solved by standard optimization methods (sometimes the solution can even be computed analytically). Computing $(K_h)_{j,i}$ is equivalent to solve the associated control problem in a small time horizon $\delta$. An approximation method proposed in \cite{akian2008max} is to use the Hamiltonian of the problem. Alternatively, a direct
method can be used, see e.g.~\cite{BocopExamples} for background on direct methods in optimal control. After $M_h$, $K_h$ are computed (or approximated), the max-plus method works as follows: 
\begin{algorithm}[htbp]
	\caption{Max-Plus Approximation Method} 
	\label{fmalgo}
	\begin{algorithmic}[1]
		\State Discretize time horizon by $N = \frac{T}{\delta}$ steps.
		\State Choose basis functions $\{w_i\}_{1\leq i \leq p}$ and $\{z_j\}_{1\leq j \leq q}$.  
		\State Compute (or approximate) $M_h$ and $K_h$.
		\State Initialize $\lambda^T = W_h^\# \phi_T$, $v^{T,h}\todest = W_h \lambda^T$.
		\For{t = T, T-$\delta$,\dots,$\delta$ }
		\State $\lambda^{t-\delta} = M_h^\#K_h \lambda^t$.
		\State $v^{t-\delta,h}\todest = W_h \lambda^{t-\delta}$
		\EndFor  
	\end{algorithmic}
\end{algorithm}
 
 \section{Adaptive Max-Plus Approximation Method}
 In Section \ref{sec_subdomain}, we observed that to solve the optimal control  problem \eqref{problem}, we need only to focus on a neighborhood of the optimal trajectory--if we could approximately know it in advance. This also works when we intend to find an approximation of the value function. In this section, we will propose an adaptive max-plus approximation method to solve problem \eqref{problem}. The general idea is to start with a small set of basis functions, then adaptively add more basis functions to better approximate the two value functions $v_{\fromsource}$ and $v_{\todest}$ in a suitable neighborhood of the optimal trajectories, obtained from the approximate value functions.

We discretize %
the time horizon by $N = \frac{T}{\delta}$ steps. %
    Then, our algorithm consists of three main steps:
    
     \textit{Step 1. Coarse Approximation.} Let us start with a regular grid $X^H = \{ x_1,x_2,\dots, x_{p^H}\}$, with step size $H$, and let $I^H:=\{1,\ldots, p^H\}$ be the index set of $X^H$. Natural choices of basis functions and test functions, are the Lipschitz functions of the form $w_{x_i}(x) := - c\| x- x_i \|_1 $, and the quadratic functions of the form $w_{x_i}(x) := -c \| x - x_i \|^2_2 \ $, for every $x_i \in X^H$. %
         We fix some sets of basis functions $\{w_{x_i} \}_{x_i \in X^H}$ and test functions $\{z_{x_i}\}_{x_i \in X^H}$, and apply Algorithm \ref{fmalgo}. This leads to an approximation of the two value functions $v\fromsource$ and $v\todest$, by the maps $v^{t,H}\fromsource$ and $v^{t,H}\todest$, with  $t \in \{ 0, \delta,\dots,T\}$, respectively.
     
     \textit{Step 2. Optimal Trajectory Approximation.} 
For every $t \in \{ 0, \delta,\dots,T\}$ 
and every $x\in X$, we have:
     \begin{equation}\label{approx_att}
     	\begin{aligned}
     		& v^{t,H}\fromsource(x) = \max_{ 1 \leq i \leq p^H }  \{ \lambda_i^{\fromsource, t} + w_{x_i} (x) \} \ , \\
     		& v^{t,H}\todest(x) = \max_{1 \leq i \leq p^H } \{ \lambda^{\todest, t}_i + w_{x_i}(x)  \} \ . 
     	\end{aligned}
     \end{equation}
     Then, an approximation for $\F_v^t$ is given by:
     \begin{equation}\label{approx_fvt}
\F^t_{v^H} = \sup_{1 \leq i,i' \leq p^H } \{ \lambda^{\fromsource,t}_i+\lambda^{\todest,t}_{i'} + w_{x_i} + w_{x_{i'} } \} \ .
     \end{equation}
     For a given $\eta^H$, let us denote by $\Od_{\eta^H,H}^{t} \subseteq X$ the approximation of $\Od^t_{\eta^H}$ defined as follows:
     \begin{equation}\label{active}
     \Od^{t}_{\eta^H,H} = \{ x\in X \mid \F_{v^H}^t(x) > \max_{y \in X} \{ \F_{v^H}^t(y) - \eta^H \} \ \} \ .	
     \end{equation}
Denote $\Mw_{i,i'}= w_{x_i} + w_{x_{i'}}$ for all $i,i'\in I^H$.  
Using \eqref{approx_fvt}, the r.h.s.\ in \eqref{active} can be computed 
as a function of 
the scalars $\Mw^*_{i,i'} := \max_{y \in X} \Mw_{i,i'}(y)=\<w_{x_i}, w_{x_{i'}}>$:
      \begin{equation}\label{active_ii}
      	   \max_{y \in X} \F^t_{v^H}(y) = \max_{i,i'\in I^H} \{ \lambda^{\fromsource,t}_i + \lambda^{\todest,t}_{i'} + \Mw^*_{i,i'} \} \ .
      \end{equation}
Moreover, for the above basis functions $w_{x_i}$ and $w_{x_{i'}}$, %
the scalars $\Mw^*_{i,i'} $ can be computed analytically.
      Denote $\Nw_{v^H}^t(i,i') = \lambda^{\fromsource,t}_i + \lambda^{\todest,t}_{i'} + \Mw^*_{i,i'}$ and let $\Nw_{v^H}^{t,*} = \max_{i,i' \in I^H} \Nw_{v^H}^t(i,i')$. In this step, we first select the couples $(i,i')$ as follows:
      \begin{equation}\label{opt_ii}
      	\Id^t_{\eta^H,H} := \{ (i,i')\in (I^H)^2\mid \Nw_{v^H}^t(i,i') > \Nw_{v^H}^{t,*} - \eta^H \} \ .
      \end{equation}
Then, based on $\Id^t_{\eta^H,H}$, we select $\Ad^t_{\eta^H,H} \subset X$ as follows:
      	\begin{equation}\label{active_t}
      		\Ad^t_{\eta^H,H} = \{x\in X \mid \exists (i,i') \in \Id^t_{\eta^H,H} \; \Mw_{i,i'}(x) > \Mw^*_{i,i'} - \eta^H  \}\enspace .
      	\end{equation}
\todo[inline]{MA: I have put $\exists$ since 
it is impossible to show that $\Od^t_{\eta^H,H}\subset\Ad^t_{\eta^H,H}$ if we put
$\forall$. This is not logical. One would need some regularity of $\Nw_{v^H}^t(i,i') $ with respect to $i,i'$   for this to work, and this is a huge work to show that. Maybe your numerical results work because of such a regularity. In the general case one need to define $\Ad^t_{\eta^H,H}$ with ``there exists''.  }
The set $\Ad^t_{\eta^H,H}$ can be compared with $\Od^{t}_{\eta^H,H}$, and
seen as an approximation of $\Od^t_{\eta^H}$.
      At the end of this step, we obtain an \textit{active region}:
     \begin{equation}\label{activeregion}
     X_f =%
     \mathop{\cup}_{t\in\{0,\delta,\dots,T\}} \{\Ad^t_{\eta^H,H} \} \ .
     \end{equation} 
           
      \textit{Step 3. Fine Approximation.} In this step, we consider the discretization of  $X$ by a regular grid $X^h$ with step size $h <H$,
and set $X^h_f := X^h\cap X_f=\{x_1, x_2, \dots, x_{p^h} \}$.
For the purpose of efficiency, we shall directly compute
      	\begin{equation}\label{active_th}
      		\Ad^{t,h}_{\eta^H,H} =  X^h\cap \Ad^{t}_{\eta^H,H}\enspace,
      	\end{equation}
then $X_f^h = \mathop{\cup}_{t\in\{0,\delta,\dots,T\}} \{\Ad^{t,h}_{\eta^H,H} \}$.
Given $X^h_f$, we add more basis functions and test functions by using the points in $X^h_f$. We then use the new set of basis functions: $\{w_{x_i} \}_{x_i \in (X^H \cup X^h_f)}$ and the new set of test functions: $\{z_{x_i}\}_{x_i \in (X^H \cup X^h_f)}$ to approximate the two value functions at each time step $t$.
      
The above approximation steps can be easily repeated, for instance, $m$ times. Moreover, the discretization grids need not be regular, and in fact, the general error estimate established in \cite{akian2008max} applies to an irregular grid.
The error is expressed in terms of an abstract mesh parameter, defined as the maximal diameter of a cell of the Voronoi tesselation induced by the grid points of the \textit{active region}. 

To define the repeated steps, we need a family of parameters $\{ \eta_l \}_{l =1,2,\dots,m}$ selecting the \textit{active regions} based on the previous two directions' approximations. We also need a family of mesh steps $H_1 > H_2 > \cdots >H_{m+1}$ and the corresponding discretization grids $X^{H_1},\ldots , X^{H_{m+1}}$ of $X$, for constructing the space discretization of the \textit{active regions}. 

We assume these parameters are fixed in advance.
      Then, we get the following algorithm:
      \begin{algorithm}[htbp]
	\caption{Adaptive Max-Plus Approximation Method} 
	\label{afmalgo}
	\begin{algorithmic}[2]
		\State Discretize time horizon by $N = \frac{T}{\delta}$ steps.
		\State Set \Ba \ and \Te \ to empty sets, set $X_f^{H_1}$ to $X^{H_1}$;
		\For {$l = 1$ to $m+1$}
		\State \Ba \  = \Ba \ $\cup \ \{w_{x_i}\}_{x_i\in X_f^{l}}$;
		\State \Te \ = \Te \ $\cup \ \{z_{x_i}\}_{x_i\in X_f^{l}}$;
		\State Approximate $v\fromsource $, $v\todest$ using Algorithm \ref{fmalgo} with \Ba \ and \Te;
                \If{$l\leq m$}
		\State Set $H=H_l$, $\eta^H=\eta_l$;
                \State Set $I^H$ as an index set for \Ba;
                \State Set $h=H_{l+1}$, and $X_f^{H_{l+1}}$ to empty set;
		\For{$t=0,\delta,\dots,T$}
		\State Compute $\Mw_{i,i'}^*$ for all $i,i'\in I^H$;
		\State Compute $\Id^t_{\eta^H,H}$ as in \eqref{opt_ii};
		\State Compute $\Ad_{\eta^H,H}^{t,h}$ by {\rm (\ref{active_t},\ref{active_th})};
		\State $X_f^{H_{l+1}} = X_f^{H_{l+1}} \cup \Ad_{\eta_H,H}^{t,h}$ ;
		\EndFor
                \EndIf
		\EndFor   
	\end{algorithmic}
\end{algorithm}
\todo[inline]{MA: in the algo, one cannot compute $\Ad^t_{\eta_l}$ as in \eqref{active_t} since this is an infinite set, one should select directly the discretization points, then the first step  ``Discretize $X_f$ by $X_f^{H_l}$ with step size $H_l$'' was useless, so either one begin by first and continue or do the two levels and at the end do not compute the active set, this is what I have done above. Also the best is to cmpute directly the fine points so I replaced $\Ad^t_{\eta_H}$ by $\Ad_{\eta_H}^{t,h}$ as in \eqref{active_th}}

We count, in Algorithm \ref{afmalgo}, each time's computation of one \textit{level} $l$, that is the first two main steps above  when $H=H_l$ and the discretization part of the last one when $h=H_{l+1}$. \todo[inline]{MA: I changed the words, but the above paragraph has to be put after no ? SG: I dont understand the English of this sentence....}

For each level $l\in\{1,2,\dots,m+1\}$ of Algorithm~\ref{afmalgo}, let  $v^{t,H_l}\fromsource, v^{t,H_l}\todest$ with $t\in \{0,\delta,\ldots, T\}$, be the approximations of $v^t\fromsource$ and $v^t\todest$ computed using the
discrete \textit{active region} $X_f^{H_l}$, and for $l\leq m$, let us denote by $X^{(l+1)}_f$ the union of the \textit{active regions} $\Ad^{t}_{\eta_l,H_l}$ at time $t$, selected by  {\rm (\ref{opt_ii},\ref{active_t})}, using $v^{H_l}\fromsource, v^{H_l}\todest$ and $\eta_{l}$. We also set $X^{(1)}_f=X$.
 Then, for all $l$, the set $X_f^{H_l}$ is the discretization of $X_f^{(l)}$.
For all $l\in\{1,2,\dots,m+1\}$ and $t\in \{0,\delta,\ldots, T\}$,  let us denote by $\tilde{v}^{t,H_l}\fromsource$ and $\tilde{v}^{t,H_l}\todest$ the approximations of $v^t\fromsource$ and $v^t\todest$ using Algorithm \ref{fmalgo} with the sets of basis functions and test functions obtained from the discretization grids $X^{H_l}$ of $X$ with mesh step $H_l$. Due to the initialization, the functions $\tilde{v}^{t,H_l}\fromsource,\tilde{v}^{t,H_l}\todest$ coincide with $v^{t,H_l}\fromsource, v^{t,H_l}\todest$ for $l=1$.
We have the following result:

\begin{theorem}\label{conver}$ $
	\begin{itemize}
		\item[(i)] For every $l \in \{1,2,\dots,m \}$, there exists an $\bar{\eta}_l$ depending on $H_l$ and $\delta$ such that for all $\eta_l \geq \bar{\eta}_l$, and $t\in \{0,\delta,\ldots, T\}$, 
$X^{l+1}_f$ contains $\Gamma^*_t$, that is the set of geodesic points for problem \eqref{problem} at time $t$.
		\item[(ii)] Take $\eta_l$ as proposed in {\rm (i)}, then for every $l\in\{2,\dots,m+1\}$, $t\in \{0,\delta,\ldots, T\}$ and $x\in \Gamma^*_t$,  we have $v^{t,H_l}\fromsource(x) = \tilde{v}^{t,H_l}\fromsource(x), v^{t,H_l}\todest(x) = \tilde{v}^{t,H_l}\todest$(x). Thus, $\{v^{t,H_m}\fromsource\}$, $\{v^{t,H_m}\todest\}$ converge to $v^t\fromsource, v^t\todest$ respectively as $H_m \to 0$.  
	\end{itemize}
\end{theorem}
\todo[inline]{MA: as stated before, the result was false (even if $\Od^{t}_{\eta^H,H} \subset \Ad^t_{\eta^H,H}$, see above). Indeed, to get the equality for all $t$ one need that $\delta$ is small and less than $\eta$, to handle the approximation of $v$ by a discretization in time with step $\delta$. Since we want to apply the result for $\delta$ not too small, one need to state the equalities for the discrete times only. This is why I defined $\Gamma^*_t$ instead of
$\Gamma^*$.}
\begin{proof}
  We give the proof in the two-level case (the extension to the multi-level case follows along the same lines). Fix a time step $\delta$, and a time $t\in \{0,\delta,\ldots, T\}$. We first notice that $\Od^{t}_{\eta^H,H} \subset \Ad^t_{\eta^H,H}$.
As shown in Proposition~\ref{pro_opt}, the value function in a geodesic point $x\in \Gamma^*_t$ satisfies $\F_v^t(x) = \sup_{y \in X} \F_v^t(y)$. We know  that the approximations $\tilde{v}^{t,H}\fromsource$ and $\tilde{v}^{t,H}\todest$ have certain error bounds (for the sup-norm) $\varepsilon^H\fromsource, \varepsilon^H\todest$ resp., depending on $H$ and $\delta$,  but not on $t$:
\[ \|v^{t,H}\fromsource - v^t\fromsource \| \leq \varepsilon^H\fromsource, \ \|v^{t,H}\todest - v^t\todest \| \leq \varepsilon^H\todest \ . \]   
Denote $\varepsilon^H = \varepsilon^H\fromsource+\varepsilon^H\todest$,  we have for every $y\in X$:
\[\ (\F_v^t(y) - \varepsilon^H) \leq \F_{v^H}^t(y) \leq (\F_v^t(y) + \varepsilon^H) \ .\]
  Consider now $x' \notin X_f:=\cup_{t\in \{0,\delta,\ldots,T\}}\Ad^t_{\eta^H,H}$, 
so that $x'\notin \Ad^t_{\eta^H,H}$, and so $x'\notin \Od^{t}_{\eta^H,H}$. Then
\begin{equation*}
	\begin{aligned}
		\F_{v}^t(x') &\leq \F_{v^H}^t(x') + \varepsilon^H \leq \sup_{y\in X} \{(\F_{v^H}^t(y) - \eta^H) +\varepsilon^H\} \\ 
		& \leq \sup_{y\in X} \{ \F_v^t(y) + ( 2 \varepsilon^H - \eta^H ) \} \ .
	\end{aligned}
\end{equation*}
Thus, if we take $\eta^H$ big enough such that  ($2 \varepsilon^H - \eta^H) < 0$, we have $x' \notin \Gamma^*_t$, and the result of (i) follows.

 The result in (ii) is then straightforward, using Proposition \ref{pro_value}. 
\end{proof}

\section{Error Analysis and Computational Complexity}
\todo[inline]{SG: I deleted ``ideal'' before computational complexity, it is not so clear}
In this section, we will analyze the computational complexity of our algorithm, and give the optimal parameters to turn the algorithm.

Let us start with evaluating the neighborhood of the optimal trajectory:
\todo[inline,color=red]{SG: caveat!
  saying:
  ``for every $x\in \Od^t_{\eta}$, there exists a $x^* \in \Gamma^*$ and $\beta>0$ '' is weak as $\beta$ is allowed to depend on $x$, dont we need that $\beta$ be independent on $x$? I rewrote.}
\begin{proposition}\label{pro_neighbor}
For every $t\in[0,T]$ and for every $x\in \Od^t_{\eta}$, there exists a $x^* \in \Gamma^*_t$ and such that:
	\[ \|x-x^*\| \leq C (\eta)^\beta \ , \]
	where $C>0$ and $\beta>0$ are constants independent
        of $x$, $t$ and $\eta$.
\end{proposition}
\todo[inline]{SG: I rewrote the next para}
In Proposition \ref{pro_neighbor}, the exponent $\beta$ determines
the growth of the neighborhood $\Od_{\eta}$ of the optimal trajectories,
as a function of $\eta$. This exponent depends
on the geometry of the value function. We shall see in Proposition~\ref{pro_beta} that for typical instances, taking $\beta=1/2$ is admissible.

Based on Proposition~\ref{pro_neighbor},
and the property that $\Od^{t}_{\eta^H,H} \subset \Ad^t_{\eta^H,H}\subset \Od^{t}_{2\eta^H,H}$ are approximations of $\Od^{t}_{\eta^H}$,
  we obtain the following general space complexity result:
\begin{proposition}\label{space_complexity}
  Given the sets of parameters $\{\eta_l\}_{l=1,2,\dots,m}$ and $\{H_l\}_{l=1,2,\dots,m+1}$, the number of discretization points generated
  by the adaptative max-plus approximation method can be bounded
  as follows:
	\begin{equation}\label{com_space}
		\Space(\{\eta_l,H_l\}) = O\Big( \big(\frac{1}{H_1}\big)^d +  \sum_{l=2}^{m+1} \big(\frac{(\eta_{l-1})^{\beta(d-1)}}{(H_l)^d}\big) \Big) . 
		\end{equation}
\end{proposition}
\begin{skproof}
  The summand $(\frac{1}{H_1})^d$ is the number of discretization points needed in the first level's grid, for which we discretized using mesh step $H_1$. Each summand $\big(\frac{(\eta_{l-1})^{\beta(d-1)}}{(H_l)^d}\big) $ corresponds to the number of points in the level-$l$'s grid, which is a "tubular" neighborhood around the optimal trajectory: at each time step, we only approximate the value functions using the points in a ball with radius $(\eta_{l-1})^\beta$ around the optimal trajectory. (This idea of using tubular neighborhoods of optimal paths to obtain complexity estimates originates from our recent work~\cite{akian2023multi}, dealing with a  minimal time optimal control problem.)
\end{skproof}
\todo[inline]{SG: I explained the link with our mtns article above}

To obtain a complexity bound showing an attenuation of the curse of dimensionality, we certainly do not want the
value function to be too ``flat'' near optimal trajectories. Indeed,
this would result in a large neighborhood $\Od_\eta$, and since
this neighborhood is used to reduce the search space and define the new grid
in Algorithm~\ref{afmalgo}, the size of the new grid would not be so much
reduced. Therefore,
we make the following convexity assumption, around the optimal trajectories.
\todo[inline]{SG: I re-explained above. I also rewrote the assumption.}
\begin{assumption}\label{assump_beta}
The functions
  $v^t\fromsource$ and $v^t\todest$ are $\mu-$strongly concave and $X$ is a convex set.
\end{assumption}
\todo[inline]{MA: I change strongly convex in strongly concave since otherwise the inequality below does not help if $x\in \Od_\eta^t$. I also changed the proof by taking another point than $2x^*-x$! and eliminated $L$ since I am not sure that this works one should need first to prove that $\Od^t_\eta$ is in a neighborhood of $\gamma^*$...}
\begin{proposition}\label{pro_beta}
	Under Assumption \ref{assump_beta}, we can take $\beta = \frac{1}{2}$ in Proposition~\ref{pro_neighbor}. %
\end{proposition}
\begin{proof}
For all $t\in [0,T]$, the function $\F^t_v$ is $2\mu-$strongly concave on $X$.
Let $x\in \Od^t_\eta$ and let $x^*\in \Gamma^*_t$.
For all $s\in [0,1]$, the point $s x+(1-s) x^* \in X$. %
Then, by the strong concavity property, we have
	\begin{equation*}
		\begin{aligned}
			&\F_v^t(s x+(1-s) x^*) +\mu (s x+(1-s) x^*)^2 \\
			& \quad \geq \frac{1}{2} \left\{ s (\F_v^t(x) + \mu x^2) +(1-s) (\F_v^t(x^*) +\mu (x^*)^2 ) \right\} \ .
		\end{aligned}
	\end{equation*}  
	By a simple computation we obtain that if $s>0$, then
 $\|x-x^* \| \leq (\frac{\eta}{2 \mu (1-s)})^{\frac{1}{2}}$,
and passing to the limit in $s$, we deduce that $\|x-x^* \| \leq (\frac{\eta}{2 \mu})^{\frac{1}{2}}$.
\end{proof}

To make sure our \textit{active region} $X_f$ does contain all $\Gamma^*_t$,
with $t=0,\delta, \ldots, T$, we need to take $\eta_l$ big enough, as discussed in Theorem \ref{conver}.
\todo[inline]{SG: I do not understand the following sentence (what is the total error?): Thus, we need first know exactly the total error.}

 Let us first focus on the approximation error, that is the approximation of $S^\delta[w]$. For every basis function $w_i$ and test function $z_i$, we have:
    \begin{equation}\label{small_appro}
    \begin{aligned}
    	&\< z_i , S^\delta [w_i] > =  \\
    	 &  \max \left\{z_i(x(0)) + \int_{0}^{\delta} \ell( x(s),u(s) ) ds +  w_i(x(\delta)) \right\} \ , 
    	    \end{aligned}
    \end{equation}
    over the set of trajectories $(x(s),u(s))$ satisfying \eqref{dynamic&constraint}. This is an optimal control problem similar to the original one,
    but with two new essential properties: first, the time horizon $\delta$ is {\em small}, and second, the initial and final costs, $z_i$ and $w_i$
    are ``nice'' concave functions, e.g., strongly concave quadratic forms.
    Then, the strong convexity of the initial or terminal cost
    ``propagates'' over a small horizon, which entails that~\eqref{small_appro}
    is actually a {\em convex} infinite dimensional optimization problem,
    which, after an appropriate discretization, using a so-called {\em direct method}
    in optimal control, can be reduced
    to a convex finite dimensional optimization problem,
    which can be solved {\em globally} by convex optimization methods
    method.  
    in which the authors used a gradient descent to compute $\< z_i , S^{\delta}[w_i] >$. 
    Alternatively, in \cite{akian2008max}, the authors approximate~\eqref{small_appro}
    using the Hamiltonian, which results in an error $\bO(\delta^2)$ or $\bO(\delta^{\frac{3}{2}})$, depending the properties of $z_i$ and $w_i$. However, to get the best complexity bounds, we need
    to assume that ~\eqref{small_appro} is approximated with a high degree of accuracy.
    Thus, we shall make the following assumption:
    \begin{assumption}\label{assump_idea_1}
    	The functions $z_i$, $w_i$ are strongly concave, and there exists a $\bar{\delta}$ such that, for every $\delta \leq \bar{\delta}$,  $\< z_i , S^{\delta}[w_i] >$ can be computed exactly, or with an error negligible compared with the projection error, by a direct method.   
    \end{assumption}
    \todo[inline,color=blue!30]{SG: I explained formally the notion of ideal complexity bound by introducing the language of Oracle model (oracle Turing machine are allowed to call an oracle which counts for one time unit)}
    This will allow us to obtain an {\em ideal} complexity bound, in an {\em oracle} Turing machine model,
    in which the time to solve a convex optimal control problem, in a small horizon,
    by calling a direct method (calling the oracle), is counted as one unit. This ideal complexity bound
    can be subsequently refined to get an effective bound in the ordinary Turing model
    of computation, recalling that $\varepsilon$-approximate solutions
    of {\em well conditioned} convex programming problems can be obtained in polynomial time
    by the ellipsoid or interior point methods. Using such an ideal model of computation
    is justified, since the only source of curse of dimensionality is the growth
    of the grid size, and since the execution time in this model is essentially
    the size of the largest grid.

    To bound the projection error, we need to make the following assumption:
    \begin{assumption}\label{assump_idea_2}
The functions  	$v^t\fromsource, v^t\todest$ are $L_v$-Lipschitz continuous, $\alpha_1$-semiconvex, $\alpha_2$-semiconcave w.r.t. $x$ for every $t\in [0,T]$. 
    \end{assumption}
    \todo[inline,color=red]{SG: caveat, it seems the $\Delta x^2$ estimate is only in the Lakhoua thesis, I did not find it in the SICON paper, so I replaced references to the SICON paper by precise references to her thesis.}
    By the result of \cite[Th.~83]{lakhouathesis}, using quadratic basis functions and Lipschitz test functions, we get a projection error $\bO(\frac{\Delta x^2}{\delta})$, where $\Delta x$ is the mesh step of the grid. Then, combining with the result of Theorem \ref{conver}, we have the following result for the total error:
    \begin{theorem}\label{error}
    	Make Assumptions \ref{assump_idea_1} and~\ref{assump_idea_2}, choose quadratic basis functions and Lipschitz test functions, choose $\delta \leq \bar{\delta}$. Then, there exists a constant $C>0$ depending on $\delta$ such that, for a given set of mesh steps $\{H_l\}_{l\in \{1,2,\dots,m+1\}}$, set $\eta_l = C (H_l)^2$, for every $l \in \{1,2,\dots,m\} $, we have:
    	\begin{equation*}
    	\begin{aligned}
    	& \| v^{H_l,t}\fromsource-v\fromsource^t \|_\infty \leq C (H_l)^2 \ , \\
    	& \| v^{H_l,t}\todest-v\todest^t         \|_\infty \leq C (H_l)^2 \  ,
    	\end{aligned}		
        \end{equation*}
        for every $l \in \{1,2,\dots,m+1 \}$.
    \end{theorem}
    \begin{proof}
      Under Assumption \ref{assump_idea_1}, we are allowed to ignore the propagation error, so that the total error is only the projection error, by~\cite{akian2008max}, which, by
      \cite[Coro.~69]{lakhouathesis},
is of order $(H_l)^2$ for each level-$l$, under assumption \ref{assump_idea_2}.
    \end{proof}
    
    Theorem \ref{error} indeed give us an upper bound for choosing the parameters
    $\eta_l$,
    depending on the parameters $H_l$.
    Let us plug this relationship between $\eta_l$ and $H_l$ into \eqref{com_space}, and use the result of proposition \ref{pro_beta} under the Assumption \ref{assump_beta}, we have
     \begin{equation*}
     \begin{aligned}
     & \Space( \{H_l\}_{l=1,\dots,m+1})  \\ 
     &\leq  O\Big( (H_1)^{-d} + C^{d-1} \sum_{l=2}^m \big( ( H_{l-1} )^{d-1} (H_l)^{-d}\big)  \Big) \ .   
     \end{aligned}
     \end{equation*}
     Suppose now we want to have a final error in the order of $\varepsilon$, then we need to take $H_{m+1}  = \bO(\varepsilon^{\frac{1}{2}})$. Once $H_{m+1}$ is fixed, $\Space$ is a convex function w.r.t. $\{H_l\}_{l=1,\dots,m}$. We also notice that, up to a multiplicative factor,
     the computational complexity, in our oracle model,
     is the same as space complexity. Then, we have the following main result for the computational complexity of our algorithm:
     \begin{theorem}\label{th_complexity} Assume Assumption \ref{assump_beta}, and take the same condition as in Theorem \ref{error}, in order to get an error $\bO(\varepsilon)$ :
     \begin{itemize}
     \item[i.] We shall take $H_m = C(\frac{1}{\varepsilon})^{\frac{1}{2}}$, and $H_{l} = C (H_m)^{\frac{l}{m}}$ for all $l \in \{1,2,\dots, m-1\}$. In this case, the total computational complexity of our $m$-level method, expressed
       in the oracle model, is bounded by $ (\frac{1}{\varepsilon})^{\frac{d}{2m}}$.
     	\item[ii.] Set $m = \lceil \frac{1}{2}|d \log (\varepsilon) |  \rceil$, and take $\{H_l\}_{l\in\{1,2,\dots,m+1\}}$ as proposed in i., then the total computational complexity reduces to $\bO(C^d (1/\varepsilon))$.
     \end{itemize}
     \end{theorem}
     \begin{skproof}
       To get a final error $\varepsilon$, by the result of Theorem \ref{error}, we need to take $H_m = (\frac{1}{\varepsilon})^{\frac{1}{2}}$. We notice that when $H_m$ is fixed, $\Space$ is a convex function w.r.t. each $H_l$. Then, by taking it's minimum w.r.t. each $H_l$ we obtain the result of $i.$
       Substituting these values into $\Space$,  further taking the minimum
      of $\Space$ w.r.t. $m$, we obtain the result of $ii.$ .
     \end{skproof}
     \todo[inline]{SG: (not critical) I am puzzled by teh constaht $C^d d$ above since $d$ can be absorbed in $C^d$ by changing $C^d$, but I am afraid of breaking things since SL you surely have a logic in writing $C^d d$ . Addendum after discussing after MA. This is OK. However, the term $dC^d$ can be bounded by $(C')^d$ as $d\to\infty$, i.e. the multiplicative factor $d$ is marginal, so In the abstract and introduction I included the simpler formulation. }

    \section{Numerical Experiments}
We applied our algorithm to a simple example, in which
the value function is known, so that the final approximation error can be computed exactly:
the linear-quadratic control problem.  

    Consider the problem \eqref{problem} with $U = \R^d$ and $X = [-5,5]^d$, the running cost $\ell(x,u) = -\|x\|^2 -\frac{1}{2}\|u\|^2$, dynamics $f(x,u) = u$, initial and final cost functions $\phi_0(x) = -\frac{1}{2}\|x-x_0\|^2$, $\phi_T(x) = -\|x-x_T\|^2$ with $x_0 = (-3,\dots,-3)$ and $x_T = (3,\dots,3)$. The time horizon is $T = 5$ and is discretized with the time step $\delta = 0.5$.

    For our algorithm, we choose quadratic basis functions and test functions with $c=10$, centered at the points of regular grids and we do two tests.
    In both, we count the number of discretization points, and so the number of basis functions, for our algorithm.
    These results have to be compared with the number of basis functions necessary for max-plus method of \cite{akian2008max}, or the number of grid points of the grid-based methods.

    For the first test, we fix the final grid mesh $h$ to $0.2$, so that the final precision is in $\bO(0.04)$,
    we first show the number of max-plus basis functions, when the dimension varies from 2 to 4,
    we give for comparison the number of grid points for an ordinary finite-difference based method:
    \begin{center}
    	\footnotesize
    	\begin{tabular}{c|c|c|c}\label{dim}
    		dimension $d$ & 2 & 3& 4 \\
    		\hline
    		$\sharp$  basis functions & 678 & 5280 & 46500 \\ 
    		\hline
    		$\sharp$ ordinary grid points   & $O(10^5)$ & $O(10^8)$ & $O(10^{10})$		
    	\end{tabular}
    \end{center}
    For the second test, we fix the dimension $d$ to 3, and make the final grid mesh vary from $0.5$ to $0.02$:
    \begin{center}
    	\footnotesize
    	\begin{tabular}{c|c|c|c|c}\label{mesh}			
    		mesh step $h$ & 0.5 & 0.2& 0.05& 0.02 \\
    		\hline 
    		$\sharp$  basis functions& 3170 & 5280 & 22490& 38970   \\
    		\hline
$\sharp$ ordinary grid points & $O(10^6)$ & $O(10^8)$ & $O(10^{10})$ & $O(10^{11})$
    	\end{tabular}
    \end{center}    
       
The algorithm has been implemented in MATLAB with some functions written in C++, and is executed on a single core of a IntelCore I7 at 2.3Gh with 16Gb RAM.
We tried different kinds of linear-quadratic control problems, changing the costs and dynamics, and the numerical tests showed results similar to the above tables. We also observed a similar growth rate of the CPU time. In all cases, the CPU time for a $4$-dimensional problem, with a final grid mesh size $0.2$, discretizing the space $[-5,5]^4$, was approximately of 80 seconds. Whereas the code is not fully optimized,
the computational speed already outperforms standard grid-based methods.
The tables show in particular that the number of basis function
grows moderately with the precision, 
consistently with the estimate of Theorem~\ref{th_complexity}, ii.
\todo[inline,color=red]{SG,MA: SL, thank you very much for the new picture. Actually, it shows a slightly super-linear growth, the function is {\em slighly convex} with the new points, whereas Theorem~\ref{th_complexity} predicts a linear growth. This needs to be further investigated. I think I was imprudent to say that experiments confirm the theorem, this is lacking maturity, so I changed for a more prudent statement, and eliminated the picture, we can keep it for the extended version of the paper, in which we will have more time to check things}

    \section{Conclusion}
    We introduced a new approximation method in optimal control, combining max-plus techniques (approximation of the value function by suprema of elementary functions) and dynamic grid refinements around optimal trajectories. This enable us to reduce the search space. In fact, under regularity assumptions, the grid size needed to obtain a $\varepsilon$-approximation
    grows linearly in $(1/\varepsilon)$, for a fixed dimension.
    We presented a first implementation, on a toy example,
    which already shows a considerable speedup by comparison with grid-based methods. We plan to refine the implementation and provide more systematic tests
    in future work.
    \todo[inline]{SG: I created a short conclusion}

\bibliographystyle{alpha}
\newcommand{\etalchar}[1]{$^{#1}$}

\appendix
\end{document}